\documentclass[11pt]{amsart}
\usepackage{eurosym}
\usepackage{amsfonts}

\newtheorem{theorem}{Theorem}[section]

\newtheorem{example}{Example}

\newtheorem{lemma}[theorem]{Lemma}

\newtheorem{remark}{Remark}

\numberwithin{equation}{section}

\begin{document}
\title[A mutual control problem for semilinear systems]{A mutual control
problem for semilinear systems via fixed point approach}
\author[R. Precup]{Radu Precup}
\address{R. Precup, Faculty of Mathematics and Computer Science and
Institute of Advanced Studies in Science and Technology, Babe\c{s}-Bolyai
University, 400084 Cluj-Napoca, Romania \& Tiberiu Popoviciu Institute of
Numerical Analysis, Romanian Academy, P.O. Box 68-1, 400110 Cluj-Napoca,
Romania}
\email{r.precup@ictp.acad.ro}
\author[A. Stan]{Andrei Stan}
\address{A. Stan, Department of Mathematics, Babe\c{s}-Bolyai University,
400084 Cluj-Napoca, Romania \& Tiberiu Popoviciu Institute of Numerical
Analysis, Romanian Academy, P.O. Box 68-1, 400110 Cluj-Napoca, Romania}
\email{andrei.stan@ubbcluj.ro}

\begin{abstract}
In this paper, we introduce and discuss the concept of a mutual control
problem. Our analysis relies on a vector fixed-point approach based on the
fixed-point theorems of Perov, Schauder, and Avramescu. In our analysis we
employ a novel technique utilizing Bielecki equivalent norms.
\end{abstract}

\maketitle

\begin{center}
% Dedicated to Professor Sehie Park with admiration
% Dedicated to Professor Sehie Park with admiration
\end{center}

\bigskip



\section{Introduction and preliminaries}

The study of systems of abstract or concrete equations has been the subject
of research for a long time, especially from the perspective of the
existence and uniqueness of solutions. In the present paper, our objective
goes beyond merely establishing the existence of solutions; we strive to
identify solutions whose components maintain a level of control over each
other.

Let $D_{1},D_{2},C$ be sets with $C\subset D_{1}\times D_{2}$ and let $%
E:D_{1}\times D_{2}\rightarrow Z$ be a mapping, where $Z$ is a linear space.
Consider the equation $E(x,\lambda )=0_{Z}$ and the problem of finding $%
\lambda \in D_{2}$ such that the equation has a solution $x\in D_{1}$ with $%
(x,\lambda )\in C$. We say that $x$ is the state variable, $\lambda $ is the
control variable, and $C$ is the controllability domain. One way to solve
the problem (see, e.g, \cite{pn}) is to use the controllability condition to
express $\lambda $ as a function of $x$, $\lambda =N(x)$, and then find a
solution to the equation $E(x,N(x))=0_{Z}$. Alternatively, we can express $x$
as $x=N(\lambda )$ and then solve the equation $E(N(\lambda ),\lambda
)=0_{Z} $. In the first case, we solve the problem
\begin{equation*}
\left\{
\begin{array}{l}
\lambda =N\left( x\right) \\[5pt]
E\left( x,N\left( x\right) \right) =0_{Z},%
\end{array}
\right.
\end{equation*}
while in the second case, the problem
\begin{equation*}
\left\{
\begin{array}{l}
x=N\left( \lambda \right) \\[5pt]
E\left( N\left( \lambda \right) ,\lambda \right) =0_{Z}.%
\end{array}
\right.
\end{equation*}

% The aim of this paper is to discuss the solvability of systems of two
% equations with \textit{mutual controllability}.
For a system of two equations, we view each variable as a control over the
other, and we aim to find solutions that achieve equilibrium under certain
controllability conditions. We call this a \textit{mutual control} problem.
Specifically, we consider four sets $D_{1},D_{2},C_{1},C_{2}$ with $%
C_{1}\subset D_{1}\times D_{2}$ and $C_{2}\subset D_{2}\times D_{1}$, $%
E_{1},E_{2}:D_{1}\times D_{2}\rightarrow Z$ two mappings and the problem
\begin{equation*}
\begin{cases}
E_{1}(x,y)=0_{Z} \\[5pt]
E_{2}(x,y)=0_{Z} \\[5pt]
(x,y)\in C_{1},\text{ }(y,x)\in C_{2}.%
\end{cases}%
\end{equation*}
The interpretation of this problem is as follows: the first equation
describes the state $x$ controlled by $y$, and the second equation describes
the state $y$ controlled by $x$. The restriction of $(x,y)$ to $C_{1}$ and $%
(y,x)$ to $C_{2}$ represents the controllability conditions.
% variable $y$ is the
% control of the state $x$ governed by the first equation, while variable $x$
% is the control of state $y$ governed by the second equation. The
% controllability conditions on $x$ and $y$ are expressed by their membership
% in $C_{1}$ and $C_{2}$, respectively. We call this problem the mutual
% control problem.
One way to solve such a problem, is to incorporate the controllability
conditions into the equations and give the problem a fixed point formulation
\begin{equation*}
\left( x,y\right) \in \left( N_{1}(x,y),\ N_{2}(x,y)\right) ,
\end{equation*}
where $N_{1},N_{2}$ are set-valued mappings $N_{1}:D_{1}\times
D_{2}\rightarrow D_{1},$ $N_{2}:D_{1}\times D_{2}\rightarrow D_{2}.$ A
solution $\left( x,y\right) $ of this fixed point equation is said to be a
solution of the mutual control problem, while the problem is said to be
\textit{mutually controllable} if such a solution exists.

In some cases, we can do even more, namely to express $x$ and $y$ as state
variables in terms of the controls $y$ and $x$, respectively. In such
situations, one has $x\in N_{1}(y)\ $ and $y\in N_{2}(x)$, where
\begin{align*}
& N_{1}:D_{2}\rightarrow D_{1},\ N_{1}(y):=\left\{ x:\text{ }
E_{1}(x,y)=0_{_{Z}}\text{ and }(x,y)\in C_{1}\right\} , \\
& N_{2}:D_{1}\rightarrow D_{2},\ N_{2}(x):=\left\{ y:\text{ }
E_{2}(x,y)=0_{_{Z}}\text{ and }(y,x)\in C_{2}\right\} .
\end{align*}
Note that $N_{1}$ indicates how the state $x$ is controlled by $y$, and $%
N_{2}$ shows how the state $y$ is controlled by $x$. The equilibrium is
achieved if there exists $(x,y)\in D_{1}\times D_{2}$ such that
\begin{equation*}
(x,y)\in \left( N_{1}(y),N_{2}(x)\right) .
\end{equation*}

It is interesting to note the similarity between the mutual control problem
and the Nash equilibrium problem. In case of the last one, we have two
functionals $J_{1}\left( x,y\right) $ and$\ J_{2}\left( x,y\right) .$
Minimizing $J_{1}\left( .,y\right) $ on $D_{1}$ for each $y\in D_{2}$ yields
to the set of minimum points denoted $s_{1}\left( y\right) ,$ while
minimizing $J_{2}\left( x,.\right) $ on $D_{2}\ $for each $x\in D_{1}$
yields to the set of minimum points denoted $s_{2}\left( x\right) .$ A point
$\left( x,y\right) $ is a Nash equilibrium with respect the two functionals
if
\begin{equation*}
\left( x,y\right) \in \left( s_{1}\left( y\right) ,\ s_{2}\left( x\right)
\right) .
\end{equation*}%
In situations when the two functionals are differentiable, a Nash
equilibrium $\left( x,y\right) $ solves the system
\begin{equation*}
\left\{
\begin{array}{l}
J_{11}\left( x,y\right) =0 \\
J_{22}\left( x,y\right) =0,%
\end{array}%
\right. \
\end{equation*}%
where $J_{11},J_{22}$ are the derivatives of $J_{1},J_{2}$ in the first and
second variable, respectively. These equations replace the ones associated
to $E_{1},E_{2},$ while the controllability conditions are such that $x,y$
minimize $J_{1},J_{2}$ in the first and second variable, respectively, that
is
\begin{eqnarray*}
C_{1} &=&\left\{ \left( x,y\right) \in D_{1}\times D_{2}:\ x\text{ minimizes
}J_{1}\left( .,y\right) \right\} , \\
C_{2} &=&\left\{ \left( y,x\right) \in D_{2}\times D_{1}:\ y\text{ minimizes
}J_{2}\left( x,.\right) \right\} .
\end{eqnarray*}%
Therefore, in the differential case, the Nash equilibrium problem appears as
a particular case of the mutual control problem as specified above. For
further details on such results, we refer the reader to \cite{bg, park, p1,
p2, ps, ps2, s, s2, s3}.

Compared to classical control theory, where the problems involve explicit
control variables (see, e.g., \cite{c, control}), mutual control problems
feature controls that are implicit by the very form of the coupling terms.
In this way, the problem of mutual control appears to be more general.

In this paper, we discuss the solvability of a mutual control problem for a
system of semilinear first-order differential equations, namely
\begin{equation}
\left\{
\begin{array}{l}
x^{\prime }+Ax=f\left( x,y\right) \ \ \  \\[5pt]
y^{\prime }+By=g\left( x,y\right) ,\text{ on }\left[ 0,T\right] .%
\end{array}%
\right.  \label{se}
\end{equation}%
Here, $T>0$; $x,y\in \mathbb{R}^{n};$ $A,B\in M_{n\times n}\left(
%TCIMACRO{\U{211d} }%
%BeginExpansion
\mathbb{R}
%EndExpansion
\right) $ and $f,g:%
%TCIMACRO{\U{211d} }%
%BeginExpansion
\mathbb{R}
%EndExpansion
^{n}\times
%TCIMACRO{\U{211d} }%
%BeginExpansion
\mathbb{R}
%EndExpansion
^{n}\rightarrow
%TCIMACRO{\U{211d} }%
%BeginExpansion
\mathbb{R}
%EndExpansion
^{n}$ are continuous mappings.
% Each of the two equations is controlled so that at time \( T \), one has
% \begin{equation}
% x(T) = ky(T), \label{cc0}
% \end{equation}
% for some given \( k > 0 \).
\ Each of the two states, $x$ and $y$, controls the other with the aim of
achieving the equilibrium specified by the controllability condition
\begin{equation}
x(T)=ky(T),  \label{cc0}
\end{equation}%
for some given $k>0$. Thus,
\begin{eqnarray*}
C_{1} &=&\left\{ (x,y)\,:\,x(T)=ky(T)\right\} , \\
C_{2} &=&\left\{ (y,x)\,:\,y(T)=\frac{1}{k}x(T)\right\} .
\end{eqnarray*}%
Related to our system $\left( \ref{se}\right) $, we assume that only one
initial state is know, say $y(0).$ Then, based on this and the
controllability condition $(\ref{cc0}),$ we determine both the solution and
the other initial state $x\left( 0\right) .$ Note that, in this case, we are
dealing a semi-observability problem since, given a control and some partial
initial information, we must determine both the output and the necessary
initial condition $x\left( 0\right) $ that leads to this output satisfying
the control. We send to \cite{control} for more details on the observability
problems.

Such a control condition is of interest in dynamics of populations when it
expresses the requirement that at a certain moment $T$ the ratio between two
populations, for example prey and predators, should be the desired $k.$
Similarly, for the control of epidemics, it expresses the requirement that
at some time one reach a certain ratio between the infected population and
that susceptible to infection. Analogous interpretations can be given in the
case of some chemical or medical reaction models.

Our analysis relies on the basic fixed point theorems of Perov, Schauder,
and Avramescu. We present both the advantages and limitations of each
theorem. Using a vector approach based on matrices instead of constants
allows us to obtain results independent of the norm of the spaces we are
working in.

\subsection{Bielecki-type norms}

For each number $\theta \geq 0$, on the space $C([0,T];\mathbb{R}^n)$, \ we
define the Bielecki norm
\begin{equation*}
\left\vert \left\vert u\right\vert \right\vert _{\theta }=\max_{t\in \left[
0,T\right] \ }e^{-\theta t}\left\vert u\left( t\right) \right\vert .
\end{equation*}
Note that if $\theta =0$, one has $\Vert \cdot \Vert _{0}=\Vert \cdot \Vert
, $ where $\Vert \cdot \Vert $ is the uniform norm. We mention that, with a
suitable choice of $\theta $, the Bielecki-type norms allow us to relax the
restrictions on constants required by the Lipschitz or growth conditions
when using fixed point theorems.

\subsection{Matrices convergent to zero}

A square matrix $M\in M_{n\times n}\left( \mathbb{R}_{+}\right) $ is said to
be convergent to zero if its power $M^{k}$ tends to the zero matrix as $%
k\rightarrow \infty $. The next lemma provides equivalent conditions for a
square matrix to be convergent to zero (see, e.g., \cite{berman,pmcz,pc}).

\begin{lemma}
\label{caracterizare} Let $A\in M_{n\times n}\left( \mathbb{R}_{+}\right) $
be a square matrix. The following statements are equivalent:
\end{lemma}

\begin{description}
\item[(a)] The matrix $A$ is convergent to zero.

\item[(b)] The spectral radius of $A$ is less than $1$, i.e., $\rho (A)<1$.

\item[(c)] The matrix $I-A,$ where $I$ is the unit matrix of the same size,
is invertible and its inverse has nonnegative entries, i.e., $\left(
I-A\right) ^{-1}\in M_{n\times n}\left( \mathbb{R}_{+}\right) .$

\item[(d)] In case $n=2$, a matrix $A=[a_{ij}]_{1\leq i,j\leq 2}\in
M_{2\times 2}\left( \mathbb{R}_{+}\right) $ is convergent to zero if and
only if
\begin{equation*}
a_{11},\ a_{22}<1\ \ \text{and\ \ }a_{11}+a_{22}<1+a_{11}a_{22}-a_{12}a_{21}.
\end{equation*}
\end{description}

In what follows, the vectors in $\mathbb{R}^n$ are treated as columns. By $%
\left\vert x \right\vert$, we mean the Euclidean norm of the vector $x$; $%
\left\langle x, y \right\rangle$ represents the inner product in $\mathbb{R}
^n$; and $\left\Vert x \right\Vert$ stands for the uniform norm in $C\left(
[0, T]; \mathbb{R}^n \right)$. For a square matrix $A$, we denote $S_A(t) =
e^{-tA}$ ($t \in \mathbb{R}$). Additionally, let $C_A$ be a upper bound of
the norms of the linear operators $e^{-tA}$ on $[-T, T]$, i.e.,
\begin{equation*}
\max_{t\in [-T,T]}\left \vert e^{-tA} \right \vert\leq C_A.
\end{equation*}

We conclude this introduction by recalling two less known fixed point
theorems which are applied in order to establish our results.

\begin{theorem}[Perov]
\label{perov} Let $\left( X_{i},d_{i}\right) ,$ $i=1,2$ be complete metric
spaces and $N_{i}:X_{1}\times X_{2}\rightarrow X_{i}$ be two mappings for
which there exists a square matrix $M$ of size two with nonnegative entries
and the spectral radius $\rho \left( M\right) <1$ such that the following
vector inequality
\begin{equation*}
\left(
\begin{array}{c}
d_{1}\left( N_{1}\left( x,y\right) ,N_{1}\left( u,v\right) \right) \\[5pt]
d_{2}\left( N_{2}\left( x,y\right) ,N_{2}\left( u,v\right) \right)%
\end{array}%
\right) \leq M\left(
\begin{array}{c}
d_{1}\left( x,y\right) \\[5pt]
d_{2}\left( u,v\right)%
\end{array}%
\right)
\end{equation*}%
holds for all $\left( x,y\right) ,\left( u,v\right) \in X_{1}\times X_{2}.$
Then, there exists a unique point $\left( x,y\right) \in X_{1}\times X_{2}$
with
\begin{equation*}
(x,y)=\left( N_{1}\left( x,y\right) ,\ N_{2}\left( x,y\right) \right) .
\end{equation*}
\end{theorem}

\begin{theorem}[Avramescu]
Let $D_{1}$ be a closed convex subset of a normed space$\ Y$, $(D_{2},d)$ a
complete metric space, and $N_{i}:D_{1}\times D_{2}\rightarrow D_{i}$, $%
i=1,2 $ be continuous mappings. Assume that the following conditions are
satisfied:

\begin{enumerate}
\item[(i)] $N_{1}(D_{1}\times D_{2})$ is a relatively compact subset of $Y$ ;

\item[(ii)] There is a constant $L\in \lbrack 0,1)$ such that
\begin{equation*}
d(N_{2}(x,y),N_{2}(x,\bar{y}))\leq L \, d(y,\overline{y}),
\end{equation*}
for all $y,\overline{y}\in D_{2}$;
\end{enumerate}

Then, there exists $(x,y)\in D_{1}\times D_{2}$ such that
\begin{equation*}
N_{1}(x,y)=x,\quad N_{2}(x,y)=y.
\end{equation*}
\end{theorem}

Some reference works in fixed point theory are the books \cite{d} and \cite%
{g}.

\section{Controllability of the mutual control problem}

In the first part of this section, we present some properties of a
particular type of matrix that will be used throughout this paper.

\subsection{On a type of matrices convergent to zero}

In the subsequent analysis based on the fixed point arguments, it is
advantageous to use the matrix
\begin{equation*}
M(\theta )=%
\begin{bmatrix}
a_{11} & a_{12}\frac{e^{\theta T}-1}{\theta } \\[3pt]
a_{21} & a_{22}\frac{1-e^{-\theta T}}{\theta }%
\end{bmatrix}%
,
\end{equation*}%
where $\theta \geq 0,$ and we need to find $\theta $ such that $M(\theta )$
is convergent to zero. Here, $a_{ij}\,(i,j=1,2)$ are nonnegative numbers
with $\ $%
\begin{equation*}
a_{11}<1\ \ \ \text{and\ \ \ }a_{22}<\frac{1}{T}.
\end{equation*}%
Notice that the last inequality guarantees $a_{22}\frac{1-e^{-\theta T}}{%
\theta }<1$ since $\frac{1-e^{-\theta T}}{\theta }\leq T$ for every $\theta
\geq 0.$

From Lemma \ref{caracterizare}, the matrix $M(\theta )$ is convergent to
zero if and only if $h(\theta )<0$, where
\begin{align*}
h(\theta )& =\text{tr}(M(\theta ))-1-\text{det}(M(\theta )) \\
& =a_{11}+a_{22}\frac{1-e^{-\theta T}}{\theta }-1-a_{11}a_{22}\frac{
1-e^{-\theta T}}{\theta }+a_{12}a_{21}\frac{e^{\theta T}-1}{\theta }\ \
\left( \theta \geq 0\right) .
\end{align*}
Denoting $\tau =a_{22}(1-a_{11})-a_{12}a_{21},$ one has
\begin{align*}
& h^{\prime }(\theta )=\frac{1}{\theta ^{2}}\left( a_{22}e^{-\theta
T}(1-a_{11})(1+\theta T)+a_{12}a_{21}(\theta T-1)e^{\theta T}-\tau \right) ,
\\
& h(0)=-\left( 1-a_{11}\right) \left( 1-a_{22}T\right) +a_{12}a_{21}T, \\
& h^{\prime }\left( 0\right) =-a_{22}(1-a_{11})\left( \frac{T^{2}}{2}
+1\right) <0,\text{ and } \\
& \lim_{\theta \rightarrow \infty }h\left( \theta \right) =\left\{
\begin{array}{l}
a_{11}-1\ \ \text{if }a_{12}a_{21}=0 \\[5pt]
+\infty \ \ \text{if }a_{12}a_{21}>0.%
\end{array}
\right.
\end{align*}
In the following lemma, we discuss conditions to guarantee the existence of
a $\theta $ such that $M(\theta )$ is convergent to zero.

\begin{lemma}
Assume $0\leq a_{11}<1$ and $0<a_{22}<\frac{1}{T}$.

\begin{itemize}
\item[(i)] If $h(0)<0$, then $M(0)$ converges to zero.

\item[(ii)] If $h(0)\geq 0$, then there exists $\theta _{1}>0$ with $%
h^{\prime }(\theta _{1})=0$ and

\begin{itemize}
\item[(a)] if $h(\theta _{1})<0$, then the matrix $M(\theta )$ converges to
zero for every $\theta $ between the zeroes of $h$ and does not converge to
zero otherwise;

\item[(b)] if $h(\theta _{1})\geq 0$, then there are no $\theta $ such that $%
M(\theta )$ converges to zero.
\end{itemize}
\end{itemize}
\end{lemma}

\begin{proof}
(i) is obvious. The next assertions are based on the convexity of the
function $h.$ To prove this, we show its second derivative is positive
everywhere on $(0,\infty )$. Indeed, one has
\begin{equation*}
h^{\prime \prime }(\theta )=\frac{1}{\theta ^{3}}\left( 2\left(
a_{22}-a_{22}a_{11}-a_{12}a_{21}\right) +a_{12}a_{21}\varphi (\theta
)-a_{22}\left( 1-a_{11}\right) \varphi (-\theta )\right) ,
\end{equation*}
where
\begin{equation*}
\varphi (\theta )=e^{\theta T}\left( \left( \theta T-1\right) ^{2}+1\right) .
\end{equation*}
Since
\begin{equation*}
\varphi (\theta )\geq 2\,\text{ \ and }\ \varphi (-\theta )\leq 2,\text{ for
all $\theta \geq 0,$ }
\end{equation*}
we find
\begin{equation*}
\theta ^{3}h^{\prime \prime }(\theta )\geq 2\left(
a_{22}-a_{22}a_{11}-a_{12}a_{21}\right) +2a_{12}a_{21}-2a_{22}(1-a_{11})=0,
\end{equation*}
which gives our conclusion.

(ii) Assume that $h(0)\geq 0$. Since $h^{\prime }(0)<0$ and $\lim_{\theta
\rightarrow \infty }h(\theta )=+\infty $, the function $h$ has a minimum at
some $\theta _{1}\in (0,+\infty )$, whence $h^{\prime }(\theta _{1})=0$.
Clearly, if $h(\theta _{1})<0$, then $h$ has two positive zeros and is
negative between them and nonnegative otherwise. If $h(\theta _{1})\geq 0$,
then $h(\theta )\geq 0$ for all $\theta \geq 0$, and thus there are no $%
\theta $ such that $M(\theta )$ converges to zero.
\end{proof}

\begin{remark}
In case $a_{22}=0,$ $M\left( \theta \right) $ is convergent to zero if and
only if
\begin{equation*}
a_{11}<1\text{ \ \ and\ \ \ }a_{12}a_{21}\frac{e^{\theta T}-1}{\theta }
<1-a_{11}.
\end{equation*}
\end{remark}

Clearly, solving the equation $h^{\prime }(\theta )=0$ analytically is
challenging. Thus, instead of $M(\theta )$, we may consider an approximate
variant, larger than $M\left( \theta \right) ,$ namely
\begin{equation*}
\widetilde{M}(\theta )=
\begin{bmatrix}
a_{11} & a_{12}\frac{e^{\theta T}-1}{\theta } \\[5pt]
a_{21} & a_{22}\frac{1}{\theta }%
\end{bmatrix}
.
\end{equation*}
Since $M\left( \theta \right) \leq \widetilde{M}(\theta )$ componentwise, if
$\widetilde{M}(\theta )$ is convergent to zero, then $M\left( \theta \right)
$ is so too. Following similar reasoning as above, the matrix $\widetilde{M}
(\theta )$ converges to zero if and only if $\ \widetilde{h}(\theta )<0$,
where
\begin{equation*}
\widetilde{h}(\theta )=a_{11}+\frac{a_{22}}{\theta }-1-\frac{a_{11}a_{22}}{
\theta }+a_{12}a_{21}\frac{e^{\theta T}-1}{\theta }\ \ \left( \theta
>0\right) .
\end{equation*}
Note that $\widetilde{h}$ cannot be extended at zero by continuity and in
general is not convex. However, under suitable conditions on $a_{ij},$ a
similar result to the previous lemma holds true. In the following, $W:%
%TCIMACRO{\U{211d} }%
%BeginExpansion
\mathbb{R}
%EndExpansion
_{+}\rightarrow
%TCIMACRO{\U{211d} }%
%BeginExpansion
\mathbb{R}
%EndExpansion
_{+}$ represents the Lambert function restricted to $%
%TCIMACRO{\U{211d} }%
%BeginExpansion
\mathbb{R}
%EndExpansion
_{+}$, i.e., the inverse of the function $ze^{z}\ \left( z\in
%TCIMACRO{\U{211d} }%
%BeginExpansion
\mathbb{R}
%EndExpansion
_{+}\right) .$ For details we send to \cite{lambert}.

\begin{lemma}
Assume that $a_{11}<1$.

\begin{description}
\item[(i)] If
\begin{equation}
\tau :=a_{22}(1-a_{11})-a_{12}a_{21}>0  \label{rp1}
\end{equation}
then the function $\Tilde{h}$ is strictly convex on $(0,\infty )$.

\item[(ii)] If
\begin{equation*}
a_{12}a_{21}Te^{\theta _{1}T}<1-a_{11},
\end{equation*}
where $\theta _{1}=\frac{1}{T}\left( W\left( \frac{\tau }{e\,a_{12}a_{21}}
\right) +1\right) $, then $\widetilde{h}(\theta _{1})<0$ and there is a
vicinity $V=\left( \sigma _{1},\sigma _{2}\right) $ of $\theta _{1}$ with $%
0<\sigma _{1}<\sigma _{2}<+\infty $ such that matrix $\widetilde{M}(\theta )$
is convergent to zero for every $\theta \in V$ and is not so for $\theta
\notin V.$
\end{description}
\end{lemma}

\begin{proof}
(i) Simple computations yield
\begin{equation*}
\widetilde{h}^{\prime }(\theta )=\frac{1}{\theta ^{2}}\left(
a_{12}a_{21}e^{\theta T}\left( \theta T-1\right) -\tau \right) .
\end{equation*}
Differentiating again $h^{\prime }(\theta )$ with respect to $\theta \in
(0,\infty )$, and using (\ref{rp1}), we deduce
\begin{align*}
\Tilde{h}^{\prime \prime }(\theta )& =\frac{1}{\theta ^{3}}\left(
a_{12}a_{21}e^{\theta T}\theta ^{2}T^{2}-2\left( -\tau
+a_{12}a_{21}e^{\theta T}\theta T-a_{12}a_{21}e^{\theta T}\right) \right) \\
& =\frac{a_{12}a_{21}}{\theta ^{3}}e^{\theta T}\left( \theta
^{2}T^{2}-2\theta T+2\right) +\frac{2\tau }{\theta ^{3}}\left(
-a_{11}a_{22}-a_{12}a_{21}+a_{22}\right) \\
& =\frac{a_{12}a_{21}}{\theta ^{3}}e^{\theta T}\left( 1+(\theta
T-1)^{2}\right) +\frac{2\tau }{\theta ^{3}}>0.
\end{align*}

(ii) Note that from $a_{22}(1-a_{11})>0$, we find $\lim_{\theta \rightarrow
0}\widetilde{h}(\theta )=+\infty $, while from $a_{12}a_{21}>0$, we have $%
\lim_{\theta \rightarrow \infty }\widetilde{h}(\theta )=+\infty $.
Therefore, $\widetilde{h}$ has a minimum $\theta _{1}\in (0,\infty )$ .
Since $\widetilde{h}$ is convex, we have $\widetilde{h}^{\prime }(\theta
_{1})=0$, which leads to
\begin{equation*}
a_{12}a_{21}e^{\theta _{1}T}(\theta _{1}T-1)=\tau .
\end{equation*}
Letting $z:=\theta _{1}T-1$, we obtain
\begin{equation*}
ze^{z}=\frac{\tau }{e\,a_{12}a_{21}},
\end{equation*}
thus $z=W\left( \frac{\tau }{e\,a_{12}a_{21}}\right) $. Consequently,
\begin{equation*}
\theta _{1}=\frac{1}{T}\left( W\left( \frac{\tau }{e\,a_{12}a_{21}}\right)
+1\right) .
\end{equation*}
Next, evaluating $\widetilde{h}$ at $\theta _{1}$, we find
\begin{equation*}
\widetilde{h}(\theta _{1})=a_{11}-1+a_{12}a_{21}Te^{\theta _{1}T}.
\end{equation*}
Therefore, under our assumption, $\widetilde{h}(\theta _{1})<0$. The
conclusion follows from the convexity of $\widetilde{h}$ and its limits as $%
\theta $ approaches 0 and infinity.
\end{proof}

\begin{remark}
The advantage of using $\widetilde{M}(\theta )$ instead of $M(\theta )$ is
that we have an analytical expression of the point where the minimum of $%
\tilde{h}$ is attained. However, in applications, due to numerical computer
power, it may be more advantageous to find an approximate solution of the
equation $h^{\prime }(\theta )=0$, and evaluate $h$ at that point.
\end{remark}

Since $\widetilde{M}(\theta )$ is an approximation of $M(\theta )$, there
are cases where $\widetilde{M}\left( \theta \right) $ does not converge to
zero for any $\theta >0$, however, there exists a $\theta _{0}>0$ such that $%
M(\theta _{0})$ does. The example below illustrates this situation.

\begin{example}
Let
\begin{equation*}
a_{11}=0.3,\,a_{12}=0.62,\,a_{21}=0.45,\,a_{22}=0.63,\text{ and }T=0.98.
\end{equation*}
Simple computations show that $\tau =0.162>0$ and $\theta _{1}\approx 1.1931$
, but conditions (ii) is not satisfied, since
\begin{equation*}
0.88\approx a_{12}a_{21}Te^{\theta _{1}T}>1-a_{11}=0.7.
\end{equation*}
On the other hand, solving numerically $h^{\prime }(\theta )=0$ gives an
approximate solution $\theta _{0}\approx 0.3505$, and hence $h(\theta
_{0})\approx -0.00798<0,$ i.e., $M\left( \theta _{0}\right) $ is convergent
to zero. Additionally, note that $h(0)=0.0056>0$, i.e., $M(0)$ is not
convergent to zero.
\end{example}

\subsection{Fixed point formulation of the mutual control problem}

Throughout this section, let $\beta \in \mathbb{R}^{n}\ $be an arbitrarily
given value, and define
\begin{equation*}
C_{\beta }([0,T];\mathbb{R}^{n}):=\left\{ y\in C([0,T];\mathbb{R}
^{n})\;:\;y(0)=\beta \right\} .
\end{equation*}
Additionally, we denote
\begin{equation*}
X_{\beta }=C([0,T];\mathbb{R}^{n})\times C_{\beta }([0,T];\mathbb{R}^{n}),
\end{equation*}
which is a complete metric space with respect to the metric inherited from $%
C([0,T];\mathbb{R}^{n})\times C([0,T];\mathbb{R}^{n})$. Under the initial
condition $y\left( 0\right) =\beta ,$ the system (\ref{se}) is equivalent to
\begin{equation}
\begin{cases}
x\left( t\right) =S_{A}\left( t\right) x(0)+\int_{0}^{t}S_{A}\left(
t-s\right) f\left( x(s),y(s)\right) ds \\[5pt]
y\left( t\right) =S_{B}\left( t\right) \beta +\int_{0}^{t}S_{B}\left(
t-s\right) g\left( x(s),y(s)\right) ds,%
\end{cases}
\label{semi-observability}
\end{equation}
while the controllability condition (\ref{cc0}) becomes
\begin{align}
x(0)& =kS_{A}(-T)S_{B}\left( T\right) \beta
+k\int_{0}^{T}S_{A}(-T)S_{B}\left( T-s\right) g\left( x(s),y(s)\right) ds
\label{cc1} \\
& \quad -\int_{0}^{T}S_{A}\left( -s\right) f\left( x\left( s\right) ,y\left(
s\right) \right) ds.  \notag
\end{align}
Thus, substituting relation \eqref{cc1} in \eqref{semi-observability}, for
every $t\in \lbrack 0,T]$, we have
\begin{equation}
\begin{cases}
x\left( t\right) =N_{1}(x,y)(t) \\
y\left( t\right) =N_{2}(x,y)(t),%
\end{cases}
\label{MS2}
\end{equation}
where $\left( N_{1},N_{2}\right) \colon X_{\beta }\rightarrow X_{\beta }$,
\begin{align*}
N_{1}(x,y)(t)& =kS_{A}(t-T)S_{B}\left( T\right) \beta
+k\int_{0}^{T}S_{A}(t-T)S_{B}\left( T-s\right) g\left( x(s),y(s)\right) ds \\
& \quad -\int_{t}^{T}S_{A}\left( t-s\right) f\left( x(s),y(s)\right) ds,
\end{align*}
and
\begin{equation*}
N_{2}(x,y)(t)=S_{B}\left( t\right) \beta +\int_{0}^{t}S_{B}\left( t-s\right)
g\left( x(s),y(s)\right) ds.
\end{equation*}
Clearly, the operator $(N_{1},N_{2})$ is well defined from $X_{\beta }~$to $%
X_{\beta }$. Consequently, the system \eqref{MS2} can be viewed as a
fixed-point equation in $X_{\beta }$ for the operator $(N_{1},N_{2})$.

The next result establishes the equivalence between the fixed points of the
operator $(N_{1},N_{2})$ and the solutions of our mutual control problem.

\begin{lemma}
%TCIMACRO{\TeXButton{lms}{\label{lms}} }%
%BeginExpansion
\label{lms}
%EndExpansion
A pair $\left( x,y\right) \in X_{\beta }$ is a solution of the mutual
control problem (\ref{se})-(\ref{cc0}), i.e., it satisfies both $\left( \ref%
{semi-observability}\right) $ and $\left( \ref{cc1}\right) ,$ if and only if
it is a fixed point for the operator $\left( N_{1},N_{2}\right) .$
\end{lemma}

\begin{proof}
The necessity has already been explained. For the sufficiency, assume $%
(x,y)\in X_{\beta }$ is a fixed point of $(N_{1},N_{2})$, i.e., it satisfies %
\eqref{MS2}. Letting $t=0$ in the first relation, we derive the
controllability condition \eqref{cc1}. Next, applying the operator $S_{A}(t)$
to \eqref{cc1}, we deduce
\begin{align*}
& S_{A}(t)x(0)+\int_{0}^{T}S_{A}(t-s)f(x(s),y(s))ds\,ds \\
& =kS_{A}(t-T)S_{B}(T)\beta
+k\int_{0}^{T}S_{A}(t-T)S_{B}(T-s)g(x(s),y(s))\,ds,
\end{align*}
which, when used in the definition of $N_{1}$ leads to the first relation of %
\eqref{semi-observability}, i.e.,
\begin{equation*}
x(t)=S_{A}(t)x(0)+\int_{0}^{t}S_{A}(t-s)f(x(s),y(s))\,ds,
\end{equation*}
The second relation in \eqref{semi-observability} follows directly from the
definition of $N_{2}$.
\end{proof}

\subsection{ Existence and uniqueness via Perov's fixed point theorem}

For our first existence result, we apply Perov's fixed point theorem, which
requires that $f$ and $g$ satisfy global Lipschitz conditions.

\begin{theorem}
Assume that there are constants $a,b,c,d\geq 0$ such that
\begin{eqnarray*}
\left\vert f\left( x,y\right) -f\left( \overline{x},\overline{y}\right)
\right\vert &\leq &a\left\vert x-\overline{x}\right\vert +b\left\vert y-
\overline{y}\right\vert , \\
\left\vert g\left( x,y\right) -g\left( \overline{x},\overline{y}\right)
\right\vert &\leq &c\left\vert x-\overline{x}\right\vert +d\left\vert y-
\overline{y}\right\vert ,
\end{eqnarray*}
for all $x,\overline{x},y,\overline{y}\in
%TCIMACRO{\U{211d} }%
%BeginExpansion
\mathbb{R}
%EndExpansion
^{n}.$ If there exists $\theta \geq 0$ such that the matrix
\begin{equation}
M(\theta )=\left[
\begin{array}{ll}
TC_{A}\left( a+kcC_{B}\right) & C_{A}\left( b+kdC_{B}\right) \frac{e^{\theta
T}-1}{\theta } \\[5pt]
cTC_{B} & dC_{B}\frac{1-e^{-\theta T}}{\theta }%
\end{array}
\right]  \label{Matrice Mc}
\end{equation}
is convergent to zero, then the mutual control problem \emph{(\ref{se})-(\ref%
{cc0})} has a unique solution in $X_{\beta }$.
\end{theorem}

\begin{proof}
We apply Perov's fixed point theorem to the operators $N_{1},N_{2}$ on the
space $X_{1}=C([0,T];\mathbb{R}^{n})$ and $X_{2}=C_{\beta }([0,T];\mathbb{R}
^{n})$. The first space, $X_{1}$, is equipped with the uniform norm, while
the second space, $X_{2}$, is equipped with a Bielecki norm $\Vert \cdot
\Vert _{\theta }$, where $\theta $ is given in the hypothesis. For any $x,
\overline{x}\in C\left( \left[ 0,T\right] ;%
%TCIMACRO{\U{211d} }%
%BeginExpansion
\mathbb{R}
%EndExpansion
^{n}\right) $ and $y,\overline{y}\in C_{\beta }\left( [0,T];\mathbb{R}
^{n}\right) $, one has
\begin{align*}
\left\vert N_{1}\left( x,y\right) \left( t\right) -N_{1}\left( \overline{x},
\overline{y}\right) \left( t\right) \right\vert & \leq
kC_{A}C_{B}\int_{0}^{T}\left\vert g\left( x\left( s\right) ,y\left( s\right)
\right) -g\left( \overline{x}\left( s\right) ,\overline{y}\left( s\right)
\right) ds\right\vert \\
& +C_{A}\int_{t}^{T}\left\vert f\left( x\left( s\right) ,y\left( s\right)
\right) -f\left( \overline{x}\left( s\right) ,\overline{y}\left( s\right)
\right) \right\vert ds \\
& \leq TC_{A}\left( a+kcC_{B}\right) \left\Vert x-\overline{x}\right\Vert \\
& \quad +C_{A}\left( b+kdC_{B}\right) \int_{0}^{T}e^{\theta s}e^{-\theta
s}\left\vert y(s)-\overline{y}(s)\right\vert ds, \\
& \leq a_{11}\left\Vert x-\overline{x}\right\Vert +a_{12}\frac{e^{\theta
T}-1 }{\theta }\left\Vert y-\overline{y}\right\Vert ,
\end{align*}
where
\begin{equation*}
a_{11}=TC_{A}\left( a+kcC_{B}\right) \text{ and }a_{12}=C_{A}\left(
b+kdC_{B}\right) .
\end{equation*}
Therefore, taking the supremmum over $t\in \lbrack 0,T]$ yields
\begin{equation*}
\left\Vert N_{1}(x,y)-N_{1}(\overline{x},\overline{y})\right\Vert \leq
a_{11}\Vert x-\overline{x}\Vert +a_{12}\frac{e^{\theta T}-1}{\theta }\Vert
y- \overline{y}\Vert _{\theta }.
\end{equation*}
For the second operator $N_{2}$, for every $t\in \lbrack 0,T]$, we estimate
\begin{align*}
\left\vert N_{2}\left( x,y\right) \left( t\right) -N_{2}\left( \overline{x},
\overline{y}\right) \left( t\right) \right\vert & \leq
C_{B}\int_{0}^{t}\left\vert g\left( x\left( s\right) ,y\left( s\right)
\right) -g\left( \overline{x}\left( s\right) ,\overline{y}\left( s\right)
\right) \right\vert ds \\
& \leq cTC_{B}\Vert x-\overline{x}\Vert +dC_{B}\frac{e^{t\theta }-1}{\theta }
\Vert y-\overline{y}\Vert _{\theta }.
\end{align*}
Multiplying both sides by $e^{-\theta t}$ and taking the supremum over $t\in
\lbrack 0,T]$, we obtain
\begin{equation*}
\left\Vert N_{2}\left( x,y\right) -N_{2}\left( \overline{x},\overline{y}
\right) \right\Vert _{\theta }\leq a_{21}\Vert x-\overline{x}\Vert +a_{22}
\frac{1-e^{-\theta T}}{\theta }\Vert y-\overline{y}\Vert _{\theta },
\end{equation*}
where
\begin{equation*}
a_{21}=cTC_{B}\text{ and }a_{22}=dC_{B}.
\end{equation*}
Writing the above relations in the vector form, one has
\begin{equation*}
\begin{bmatrix}
\left\Vert N_{1}\left( x,y\right) -N_{1}\left( \overline{x},\overline{y}
\right) \right\Vert \\[5pt]
\ \left\Vert N_{2}\left( x,y\right) -N_{2}\left( \overline{x},\overline{y}
\right) \right\Vert _{\theta }%
\end{bmatrix}
\leq M(\theta )
\begin{bmatrix}
\left\Vert x-\overline{x}\right\Vert \\[5pt]
\ \left\Vert y-\overline{y}\right\Vert _{\theta }%
\end{bmatrix}
.
\end{equation*}
Consequently, since the matrix $M(\theta )$ is convergent to zero, the
operator $(N_{1},N_{2})$ is a Perov contraction on $X_{\beta }$. Thus, in
view of Lemma \ref{lms}, there exists a unique solution $(x,y)\in X_{\beta }$
for the problem $\left( \ref{se}\right) $-$\left( \ref{cc0}\right) .$
\end{proof}

\subsection{Existence and localization via Schauder's fixed point theorem}

Our second result allows linear growths on the functions $f$ and $g$ instead
of the Lipschitz coditions, but at the cost of losing the uniqueness of the
fixed point for the operator $(N_{1},N_{2})$.

\begin{theorem}
\label{schauder} Assume that $f$ and $g$ satisfy the following linear growth
conditions
\begin{align}
\left\vert f(x,y)\right\vert & \leq a\left\vert x\right\vert +b\left\vert
y\right\vert +\gamma ,  \label{lgc} \\
\left\vert g(x,y)\right\vert & \leq c\left\vert x\right\vert +d\left\vert
y\right\vert +\delta ,  \notag
\end{align}
for all $x,y\in \mathbb{R}^{n}$ and some nonnegative constants $%
a,b,c,d,\gamma ,\delta $. If there exists $\theta \geq 0$ such that the
matrix $M\left( \theta \right) $ given in $\left( \ref{Matrice Mc}\right) $
is convergent to zero, then the mutual control problem \emph{(\ref{se})-(\ref%
{cc0})} has at least one solution $(x,y)\in X_{\beta }$, satisfying some
bounds of the form
\begin{equation*}
\left\vert x\left( t\right) \right\vert \leq R_{1}\quad \text{and}\quad
\left\vert y\left( t\right) \right\vert \leq e^{t\theta }R_{2},\ \left( t\in
\lbrack 0,T]\right) ,
\end{equation*}%
where $R_{1}$ and $R_{2}$ are chosen appropriately below.
\end{theorem}

\begin{proof}
We apply Schauder's fixed-point theorem to the operator $(N_{1},N_{2})$ on
the bounded, closed, and convex set $B_{R_{1}}\times B_{R_{2}}$. The sets $%
B_{R_{1}}$ and $B_{R_{2}}$ are balls centered at the origin with radii $%
R_{1} $ and $R_{2}$ in the Banach space $C([0,T];\mathbb{R}^{n})$ with the
uniform norm, and in the metric space $C_{\beta }([0,T];\mathbb{R}^{n})$
with the Bielecki norm $\left\vert |\cdot |\right\vert _{\theta }$,
respectively. Here, $R_{1}$ and $R_{2}$ are positive real numbers that will
be determined later.

To this aim, we need to show that the operator $(N_{1},N_{2})$ is completely
continuous on $X_{\beta }$, and that\ it maps the set $B_{R_{1}}\times
B_{R_{2}}$ into itself, i.e.,
\begin{equation}
\left\Vert N_{1}(x,y)\right\Vert \leq R_{1},\ \left\Vert
N_{2}(x,y)\right\Vert _{\theta }\leq R_{2}\text{ \ whenever\ }\Vert x\Vert
\leq R_{1},\ \,\Vert y\Vert _{\theta }\leq R_{2}.  \label{invarianta}
\end{equation}%
Using standard arguments, the complete continuity is immediate (see, e.g.,
\cite{pc}).

Let $\left( x,y\right) \in X_{\beta }$. Then,
\begin{align*}
\left\Vert N_{1}\left( x,y\right) \right\Vert & \leq kC_{A}C_{B}\left\vert
\beta \right\vert +kC_{A}C_{B}\int_{0}^{T}\left\vert g\left( x(s\right)
,y\left( s\right) \right\vert ds \\
& \quad +C_{A}\int_{t}^{T}\left\vert f\left( x(s\right) ,y\left( s\right)
\right\vert ds \\
& \leq TC_{A}\left( a+kcC_{B}\right) \left\Vert x\right\Vert +C_{A}\left(
b+kdC_{B}\right) \frac{e^{\theta T}-1}{\theta }\left\Vert y\right\Vert
_{\theta }+\eta _{1},
\end{align*}%
where$\ $%
\begin{equation*}
\eta _{1}=C_{A}\left( kC_{B}\left\vert \beta \right\vert +kTC_{B}\delta
+T\gamma \right) .
\end{equation*}%
Similarly,
\begin{equation}
\left\vert N_{2}\left( y,y\right) \left( t\right) \right\vert \leq
C_{B}\left\vert \beta \right\vert +C_{B}\int_{0}^{t}\left\vert g\left(
x\left( s\right) ,y\left( s\right) \right) \right\vert ds.  \label{growthN2}
\end{equation}%
Multiplying \eqref{growthN2} by $e^{-t\theta }$ and taking the supremum over
$[0,T]$ yields
\begin{equation*}
\left\vert N_{2}\left( y,y\right) \left( t\right) \right\vert \leq
cTC_{B}\left\Vert x\right\Vert +dC_{B}\frac{1-e^{-\theta T}}{\theta }%
\left\Vert y\right\Vert _{\theta }+\eta _{2},
\end{equation*}%
where
\begin{equation*}
\eta _{2}=C_{B}\left\vert \beta \right\vert +TC_{B}\delta .
\end{equation*}%
Thus
\begin{equation*}
\left[
\begin{array}{c}
\left\Vert N_{1}\left( x,y\right) \right\Vert ~~ \\[5pt]
\left\Vert N_{2}\left( y,y\right) \right\Vert _{\theta }%
\end{array}%
\right] \leq M\left( \theta \right) \left[
\begin{array}{c}
\left\Vert x\right\Vert \ \  \\[5pt]
\left\Vert y\right\Vert _{\theta }%
\end{array}%
\right] +\left[
\begin{array}{c}
\eta _{1} \\[5pt]
\eta _{2}%
\end{array}%
\right] .
\end{equation*}%
Note that the invariance condition $\left( \ref{invarianta}\right) \ $holds
if $R_{1},R_{2}$ satisfy
\begin{equation*}
M\left( \theta \right) \left[
\begin{array}{c}
R_{1} \\[5pt]
R_{2}%
\end{array}%
\right] +\left[
\begin{array}{c}
\eta _{1} \\[5pt]
\eta _{2}%
\end{array}%
\right] \leq \left[
\begin{array}{c}
R_{1} \\[5pt]
R_{2}%
\end{array}%
\right] ,
\end{equation*}%
or equivalently
\begin{equation}
\left[
\begin{array}{c}
\eta _{1} \\[5pt]
\eta _{2}%
\end{array}%
\right] \leq \left( I-M\left( \theta \right) \right) \left[
\begin{array}{c}
R_{1} \\[5pt]
R_{2}%
\end{array}%
\right] .  \label{eqmu}
\end{equation}

Since the matrix $M\left( \theta \right) $ is convergent to zero, Lemma \ref%
{caracterizare} guarantees that the matrix $\left( I-M\left( \theta \right)
\right) ^{-1}$ has nonnegative entries. Thus, we can multiply $\left( \ref%
{eqmu}\right) \ $with $\left( I-M\left( \theta \right) \right) ^{-1}$
without changing the sign of inequality. If we choose $R_{1},R_{2}$ large
enough such that
\begin{equation*}
(I-M\left( \theta \right) )^{-1}%
\begin{bmatrix}
\eta _{1} \\[5pt]
\eta _{2}%
\end{bmatrix}%
\leq
\begin{bmatrix}
R_{1} \\
R_{2}%
\end{bmatrix}%
,
\end{equation*}%
the invariance condition is satisfied. Consequently, Schauder's fixed point
theorem in $B_{R_{1}}\times B_{R_{2}}$ provides the result.
\end{proof}

\subsection{Existence via Avramescu's fixed point theorem}

Assuming that both functions $f,g$ have linear growth, and $g$ is Lipschitz
continuous with respect to the second variable $y,$ we obtain the following
existence result.

\begin{theorem}
Assume there exists constants $a,c,b,d,\gamma ,\delta \geq 0$ such that
\begin{align*}
& \left\vert f(x,y)\right\vert \leq a\left\vert x\right\vert +b\left\vert
y\right\vert +\gamma , \\
& \left\vert g(x,y)-g\left( x,\overline{y}\right) \right\vert \leq
d\left\vert y-\overline{y}\right\vert ,
\end{align*}%
and
\begin{equation*}
\left\vert g\left( x,0\right) \right\vert \leq c\left\vert x\right\vert
+\delta ,
\end{equation*}%
for all $x,y,\overline{y}\in \mathbb{R}^{n}$. If there exists $\theta \geq 0$
such that the matrix $M\left( \theta \right) ,$ given in \emph{(\ref{Matrice
Mc}),} is convergent to zero, then the mutual control problem $\left( \emph{%
\ref{se}}\right) $\emph{-}$\left( \emph{\ref{cc0}}\right) $ has at least one
solution in $X_{\beta }.$
\end{theorem}

\begin{proof}
Clearly,
\begin{align*}
& \left\vert f(x,y)\right\vert \leq a\left\vert x\right\vert +b\left\vert
y\right\vert +\gamma , \\
& \left\vert g(x,y)\right\vert \leq c\left\vert x\right\vert +d\left\vert
y\right\vert +\delta ,
\end{align*}%
for all $x,y\in \mathbb{R}^{n}$. Then, as in the proof of Theorem \ref%
{schauder}, we find $R_{1},R_{2}>0$ such that
\begin{equation*}
\left\Vert N_{1}(x,y)\right\Vert \leq R_{1}\ \ \ \text{and\ \ \ }\left\Vert
N_{2}(x,y)\right\Vert _{\theta }\leq R_{2}
\end{equation*}%
for all $\left( x,y\right) \in X_{\beta }$ with $\left\Vert x\right\Vert
\leq R_{1}$ and $\left\Vert y\right\Vert _{\theta }\leq R_{2}$. Moreover, $%
N_{1}\left( B_{R_{1}}\times B_{R_{2}}\right) $ is relatively compact in $%
C\left( [0,T];\mathbb{R}^{n}\right) $. It remains to show that $%
N_{2}(x,\cdot )$ is a contraction on $B_{R_{2}}$ with a contraction
coefficient independent of $x\in B_{R_{1}}$. For any $x\in C\left( [0,T];%
\mathbb{R}^{n}\right) $ and any $y,\overline{y}\in C_{\beta }\left( [0,T];%
\mathbb{R}^{n}\right) $, we have
\begin{equation*}
\left\vert N_{2}(x,y)(t)-N_{2}(x,\overline{y})(t)\right\vert \leq
dC_{B}\int_{0}^{t}\left\vert y(s)-\overline{y}(s)\right\vert ds,
\end{equation*}%
whence
\begin{equation*}
\left\Vert N_{2}(x,y)(t)-N_{2}(x,\overline{y})\right\Vert _{\theta }\leq
dC_{B}\frac{1-e^{-\theta T}}{\theta }\left\Vert y-\overline{y}\right\Vert
_{\theta },
\end{equation*}%
where
\begin{equation*}
dC_{B}\frac{1-e^{-\theta T}}{\theta }<1,
\end{equation*}%
since the matrix $M(\theta )$ is convergent to zero. Thus, Avramescu's
theorem applies and gives the final conclusion.
\end{proof}

\begin{remark}
Under the conditions of the previous theorem, if $\left( x,y\right) ,\
\left( x,\overline{y}\right) $ are two solutions of \ the mutual control
problem $\left( \ref{se}\right) $-$\left( \ref{cc0}\right) $, then
necessarily $y=\overline{y}.$
\end{remark}

\begin{center}
\bigskip
\end{center}

\section{\textbf{Conclusions}}

\medskip The concept of mutual control introduced in this paper and
illustrated on the case of semi-linear systems of the first order with a
final condition of proportionality of the states is likely to be used for
various other classes of systems and other controllability conditions. In
this context, problems of observability or partial observability can be
formulated. The working technique, as suggested by this paper, is mainly
based on the fixed point theory and could be completed by considering other
abstract existence principles.

\end{document}